\documentclass[11pt]{amsart}
\usepackage{fullpage}
\usepackage{amsthm}
\usepackage{amssymb}
\usepackage{ytableau}
\usepackage{tikz}
\usetikzlibrary{calc}
\usepackage{mathtools}

 \newtheorem{thm}{Theorem}[section]
 \newtheorem{cor}[thm]{Corollary}
 \newtheorem{lem}[thm]{Lemma}
 
 \theoremstyle{definition}
 \newtheorem{defn}[thm]{Definition}
 \newtheorem{rem}[thm]{Remark}
 \newtheorem{rems}[thm]{Remarks}

 \numberwithin{equation}{section}
\newcount\tableauRow\newcount\tableauCol
\newcommand\ColumnTabloid[1]{%
	\begin{tikzpicture}[scale=0.7,draw/.append style={thick,black},baseline=1mm]
		\tableauRow=0
		\foreach \Row in {#1} {
			\tableauCol=1
			\foreach\k in \Row {
				\draw($(\the\tableauCol,\the\tableauRow)+(-.5,-.5)$)--++(0,1);
				\draw($(\the\tableauCol,\the\tableauRow)+(.5,-.5)$)--++(0,1);
				\draw(\the\tableauCol,\the\tableauRow)node{\k};
				\global\advance\tableauCol by 1
			}
			\global\advance\tableauRow by -1
		}
	\end{tikzpicture}
}
\DeclareMathOperator{\Ima}{Im}
\DeclareMathOperator{\Hom}{Hom}

\DeclareMathOperator{\rank}{rank}

\begin{document}

\title{Total trades, intersection matrices and Specht modules}

\author[Maliakas]{Mihalis Maliakas}
\address{%
	Department of Mathematics\\
	University of Athens\\
	Greece}
\email{mmaliak@math.uoa.gr}

\author[Stergiopoulou]{Dimitra-Dionysia Stergiopoulou}
\address{%
	Department of Mathematics\\
	University of Athens\\
	Greece}
\curraddr{%
	Department of Mathematics\\
	University of Thessaly\\
	Greece}
\email{dstergiop@math.uoa.gr}


\begin{abstract}Trades are important objects in combinatorial design theory that may be realized as certain elements of kernels of inclusion matrices. Total trades were introduced recently by Ghorbani, Kamali and Khosravshahi who showed that over a field of characteristic zero the vector space of trades decomposes into a direct sum of spaces of total trades. In this paper, we show that the vector space spanned by the permutations of a total trade is an irreducible representation of the symmetric group. As a corollary, the previous decomposition theorem is recovered.  Also, a basis  is obtained for the module of total trades in the spirit of Specht polynomials. In the second part of the paper we consider more generally intersection matrices and determine the irreducible decompositions of their images. This generalizes previously known results concerning ranks of special cases.
\end{abstract}

\maketitle
\tableofcontents

\section{Introduction} Throughout this paper, we work over a field $\mathbb{K}$ of characteristic zero.
Let $t,k,n$ be integers such that $0 \le t < k < n$ and let $X$ be the set $\{1,2,\dots, n\}$. Suppose that $T_{+}$ and $T_{-}$ are two disjoint collections of $k$-subsets of $X$ with the property that the number of occurrences of every $t$-subset of $X$ in $T_{+}$ and $T_{-}$ are the same. Then $T=(T_{+},T_{-})$ is called a $t$-$(n,k)$ \textit{trade}.

The \textit{inclusion matrix} $W_{t,k}$ is a matrix with rows and columns indexed by $t$-subsets and $k$-subsets of $X$ respectively and the $(A,B)$ entry is equal to 1 if $A \subseteq B$ and is equal to 0 otherwise. Trades are null vectors of $W_{t,k}$. Among trades, the \textit{minimal trades} play a crucial role in design theory. A minimal $t$-$(n,k)$ trade may be represented in the form
\[(x_1-y_1)\cdots(x_{t+1}-y_{t+1})x_{t+2}\cdots x_k,\]
where $x_i, y_i \in X$ are distinct and multiplication is to be interpreted as union of sets (see Section 2.1 below). The symmetric group $\mathfrak{S}_n$ acts naturally on $X$ and thus on minimal trades. Their importance is highlighted by the classical result of Graver and Jurkat \cite{GJ} that the vector space of $t$-$(n,k)$ trades is generated by the permutations of any minimal $t$-$(n,k)$ trade if $t+k \le n$.

Recently Ghorbani, Kamali and Khosravshahi \cite{GKK} introduced the notion of \textit{total trade}. A $t$-$(n,k)$ total trade may be represented in the form \[(x_1-y_1)\cdots(x_{t+1}-y_{t+1}) \sum x_{i_{t+2}}\cdots x_{i_k},
\]
where the sum ranges over all $(k-t-1)$-subsets $\{x_{i_{t+2}}\cdots x_{i_k} \}$ of \\ $X \setminus \{x_1,y_1, \dots, x_{t+1}, y_{t+1}\}$. Let us denote by $\mathfrak{T}_{t,k,n}$ the set of $t$-$(n,k)$ total trades and by $\langle \mathfrak{T}_{t,k,n}\rangle$ the vector space generated by $\mathfrak{T}_{t,k,n}$. The following result was shown in \cite{GKK}.
\begin{thm}[{\cite[Theorem 11]{GKK}}]\label{thmintro} If $2k-1 \le n$, then the vector space of $t$-$(n,k)$ trades decomposes as the direct sum $\langle \mathfrak{T}_{t,k,n}\rangle \oplus \cdots \oplus \langle \mathfrak{T}_{k-1,k,n}\rangle$.
\end{thm}
This was used in \cite{GKK} to derive decomposition results concerning the vector space generated by permutations of a trade or permutations of a signed design. 

We recall that the irreducible modules of the symmetric group $\mathfrak{S}_n$ are naturally parameterized by partitions of $n$. For a partition $\lambda$ of $n$ we denote by $S^{\lambda}$ the corresponding irreducible module. This is called a Specht module. In our first main result we prove that $\langle \mathfrak{T}_{i,k,n}\rangle$ is an irreducible module of the symmetric group $\mathfrak{S}_n$ (see Theorem  \ref{main0}). The precise statement is the following.
\begin{thm}\label{intromain0}
	Let $t< k $ and $t+k \le n$. Then as $\mathfrak{S}_n$-modules the  subspace $\langle \mathfrak{T}_{t,k,n}\rangle$ of $M_k$ spanned by the set $\mathfrak{T}_{t,k,n}$ of $t$-$(n,k)$ total trades is isomorphic to the Specht module $S^{(n-t-1,t+1)}$.
\end{thm}
From this we obtain a basis of the $\mathbb{Z}$-module spanned by $ \mathfrak{T}_{i,k,n}$ (Corollary \ref{basis}). We also recover the result of Graver and Jurkat mentioned above. Our proofs are independent of \cite{GKK} and we recover the decomposition in Theorem \ref{thmintro} as the irreducible decomposition of the kernel of the inclusion matrix $W_{t,k}$.

Suppose $l$ is an integer such that $0 \le l \le t$. The \textit{intersection matrix} $U_{t,k,l}$ is a matrix with rows and columns indexed by $t$-subsets and $k$-subsets of $X$ respectively and the $(A,B)$ entry is equal to 1 is $|A \cap B| =l$ and is equal to $0$ otherwise. These generalize the inclusion matrices since $U_{t,k,t} =W_{t,k}$. Such matrices have been studied in different contexts in combinatorics such as design theory and association schemes. Ghareghani, Ghorbani and Mohammad-Noori \cite{GGM} studied these matrices and found an expression for their rank. The matrices $U_{t,k,l}$ as $l$ varies from $0$ to $t$ correspond to basis elements of the homomorphism space $\Hom_{\mathfrak{S}_n}(M_k,M_t)$ between the permutation modules $M_k $ and $M_t$ of the symmetric group ${\mathfrak{S}_n}$. We generalize the previous result on ranks by considering any linear operator in $\Hom_{\mathfrak{S}_n}(M_k,M_t)$. In order to state the result we need some notation.
Suppose $l$ is a nonnegative integer such that $l \le t$. For $j=0,1, \dots, t$ we define the integers \[\lambda_j(t,k,n;l) :=\sum_{s=0}^{j}(-1)^{j-s}\tbinom{j}{s}\tbinom{k-s}{l-s}\tbinom{n-k-j+s}{t-l-j+s},\]
(see Definition \ref{lambda}).
Our second main result is the following (see Corollary \ref{main3}).
\begin{thm}
	\label{intromain3}Consider a linear combination of intersection matrices $A= \sum_{l=0}^{t} c_l U_{t,k,l}$ where $c_l \in \mathbb{K}$. Then the irreducible decomposition of the image of $A$ is 
	\begin{equation*}\label{irrdec3} \bigoplus_{j \in J(t,k,n)}S^{(n-j,j)},\end{equation*} where $J(t,k,n)$ is the subset of $\{0,1, \dots, t\}$ consisting of those $j$ such that there exists $l \in \{0,1,\dots,t\}$ satisfying $c_l \neq 0$ and $\lambda_j(t,k,n;l) \neq 0$. In particular, the rank of the matrix $A$ is equal to 
	\begin{equation*}\label{introrank1} \rank(A) = \sum_{j \in J(t,k,n)}\big(\tbinom{n}{j} - \tbinom{n}{j-1} \big).\end{equation*}
\end{thm}

This is obtained from an analogous result that we prove for representations of the general linear group (see Theorem \ref{main1}).

 The second statement of the above theorem is rather surprising as the rank of a linear combination of matrices is in general not well behaved with respect to the ranks of the individual matrices. In this case, however, we have a decomposition of the images into a direct of irreducibles each of which has multiplicity equal to 1. 

The paper is organized as follows. Section 2 is devoted to background material concerning mainly representations of the symmetric group. In Section 3 we prove our first main result (see Theorem \ref{main0}) and discuss certain corollaries. The second part of the paper is independent of the first part but our techniques are similar. Preliminaries on polynomial representations of the general linear group are considered in Section 4. Our second main result (see Theorem \ref{main1} and Corollary \ref{main3}) is proved in Section 5.

\section{Background}
\subsection{Boolean algebra}Let $n$ be a positive integer and let $X$ be the set $\{1,2,\dots,n\}$. Let $2^{[X]}$ be the vector space over $\mathbb{K}$ with basis the subsets of $X$. For a nonnegative integer $k$, let $M_k$ be the subspace of $2^{[X]}$ spanned by the $k$-subsets of $X$. If $\alpha =\sum_{Y\subseteq X}\alpha_{Y}Y$ and $\beta =\sum_{Z\subseteq X}\beta_{Z}Z$ are elements of $2^{[X]}$, define their product by $\alpha \beta =\sum_{Y,Z}\alpha_Y \beta_Z(Y \cup Z)$. This makes $2^{[X]}$ a commutative ring with identity the empty set. If $\{x_1, x_2,\dots, x_k\}$ is a $k$-subset of $X$, we will denote the corresponding basis element of $M_k$ by $x_1x_2\cdots x_k$. Hence expressions such as $(x_1-y_1)\cdots(x_{t+1}-y_{t+1})x_{t+2}\cdots x_k$, where $x_i, y_j \in X$ are distinct, have a precise meaning as elements of $2^{[X]}$ and $M_k$. 

 We have the natural action of the symmetric group $\mathfrak{S}_n$ on $X$ and hence on subsets of $X$.  Then $M_k$ is the natural permutation module for $\mathfrak{S}_n$. Also we have $2^{[X]} =\oplus_{i \ge 0}M_i$.

For nonnegative integers $t \le k$, let $\psi_k^{(k-t)}: M_k \to M_t$ be the $\mathfrak{S}_n$-map defined by \begin{equation}\psi_k^{(k-t)}(x_1x_2\cdots x_k) = \sum x_{i_1}x_{i_2}\cdots x_{i_t},
	\end{equation}
	where the sum is over all $t$-subsets $\{x_{i_1}, x_{i_2}, \dots, x_{i_t}\}$ of $\{x_1,x_2,\dots,x_k\}$. It is clear that with respect to some total order on $k$-subsets and $t$-subsets of $X$ the matrix of the map $\psi_k^{(k-t)}$ is the inclusion matrix $W_{t,k}$. Such maps have been considered in different contexts, for example in the modular representation theory of the symmetric groups and related homologies \cite{Ja, FrYa, MnS, Jo}.
\subsection{Representations of the symmetric group}
In this section we recall some facts from the representation theory of the symmetric group $\mathfrak{S}_n$ over a field of characteristic zero that will be used in the sequel.

A partition of $n$ is a sequence $\lambda = (\lambda_1, \dots, \lambda_m)$ of nonnegative integers such $\lambda_1 \ge \dots \ge \lambda_m$ and $\lambda_1+\cdots+\lambda_m = n$.  Since the characteristic of $\mathbb{K}$ is zero, the irreducible representations of $\mathfrak{S}_n$ over $\mathbb{K}$ are parameterized by partitions $\lambda$ of $n$. We denote by $S^{\lambda}$  the irreducible module corresponding to $\lambda$. This is called a  Specht module. Of the various constructions of $S^{\lambda}$ that are available, we will use one that is based on column tabloids.  The facts we need may be found in \cite[Chapter 7]{F}. In order for this paper to be reasonably self contained, we present the main points below.

Since in the sequel we will need only partitions with at most two nonzero parts, our discussion will be restricted to these.

For a partition $\lambda$ of $n$, a Young tableau of shape $\lambda$ is a filling of the Young diagram of $\lambda$ with distinct entries from the set $\{1,2, \dots, n\}$. For example, a Young tableau of shape $\lambda=(5,2)$ is 
\begin{center}
$U=$ \begin{ytableau}
	2 & 7 & 1  & 4 & 6\\
	3 & 5  \\
	\end{ytableau}
\end{center}

Let us denote by $\mathcal{T}_\lambda$ the set of Young tableaux of shape $\lambda$. Let $\tilde{M}^\lambda$ be the vector space generated by $\mathcal{T}_\lambda$ subject to the relations $U+U'=0$, if $U'$ is obtained from $U \in \mathcal{T}_\lambda$ by a transposition of two entries in the same column of $U$.  The group $\mathfrak{S}_n$ acts on $\mathcal{T}_\lambda$ by permuting the entries. It follows that $\tilde{M}^\lambda$ is an $\mathfrak{S}_n$-module. For $U \in \mathcal{T}_\lambda$, we denote by $q=q_U \in \tilde{M}^\lambda$ the corresponding coset. The elements $q \in \tilde{M}^\lambda$ are called \textit{column tabloids}. It is clear that these generate $\tilde{M}^\lambda$ as a vector space. We may denote column tabloids by erasing the horizontal lines of the corresponding Young tableaux.  For example, in $\tilde{M}^{(5,2)}$ we have 
\begin{center}
	\begin{tabular}{ |c|c|c|c|c| } 
		2 & 7 & 1 & 4 & 6  \\ 
		3 & 5  \\ 
\end{tabular} \ + \  \begin{tabular}{ |c|c|c| c|c| } 
2 & \textbf{5} & 1 & 4 & 6  \\ 
3 & \textbf{7}  \\ 
\end{tabular} \ = 0. \end{center}

 We have highlighted the entries in the second column that were exchanged.
 
 Let $\lambda = (\lambda_1, \lambda_2)$ be a partition of $n$, $U$ a Young tableau of shape $\lambda$,  $q$ the corresponding column tabloid  in $\tilde{M}^{\lambda}$ and let $1 \le c \le \lambda_1-1$.  With this data we define an element  $g_{U,c} \in \tilde{M}^{\lambda}$. We distinguish two cases. \begin{itemize}
	\item Suppose $c \le \lambda_2$. Let $U_i$ be the Young tableau obtained from $U$ by exchanging the top element of column $c+1$ with the $i$th element of column $c$. Let $q_i$ be the column tabloid corresponding to $U_i$. Define $g_{U,c}=q-q_1-q_2$.
	\item Suppose $c > \lambda_2$. Let $U_3$ be the Young tableau obtained from $U$ by exchanging the two elements in columns $c$ and $c+1$. Let $q_3$ be the corresponding column tabloid. Define $g_{U,c}=q-q_3$.\end{itemize}
For example, if $c=1$ and $U$ is the Young tableau of the previous example, then $g_{U,1}$ is the element
\begin{center}
	\begin{tabular}{ |c|c|c|c|c| } 
		2 & 7 & 1 & 4 & 6  \\ 
		3 & 5  \\ 
	\end{tabular} \ - \  \begin{tabular}{ |c|c|c| c|c| } 
		\textbf{7} & \textbf{2} & 1 & 4 & 6  \\ 
		3 & 5  \\ 
	\end{tabular} - \  \begin{tabular}{ |c|c|c| c|c| } 
	2 & \textbf{3} & 1 & 4 & 6  \\ 
	\textbf{7} & 5  \\ 
	\end{tabular}. \end{center}	
The exchanged elements have been highlighted with boldface font.	
	
\begin{defn}\label{Q}
Let $Q^{\lambda}$ be the subspace of $\tilde{M}^{\lambda}$ spanned by all the $g_{U,c}$ as $U$ varies over all Young tableaux of shape $\lambda$ and $1 \le c \le \lambda_1-1$.
\end{defn} This is an $\mathfrak{S}_n$ submodule of $\tilde{M}^{\lambda}$ and we have the following well known result.
\begin{thm}\label{pres}The quotient $\tilde{M}^{\lambda} / Q^{\lambda}$ is isomorphic to the Specht module $S^{\lambda}$.
\end{thm} 
\begin{proof} Since the characteristic of $\mathbb{K}$ is equal to zero, it follows from \cite[Exercise 16, Section 7.4]{F} that our $Q^{\lambda}$ is equal to the subspace of $\tilde{M}^{\lambda}$ defined in \cite[p. 98]{F} (which is denoted there also by $Q^{\lambda}$). Then the result follows from \cite[Proposition 4, Section 7.4]{F}.\end{proof}	

Next we recall the standard basis theorem for Specht modules. Let us fix a total order on $X$. A Young tableau $U$ is called standard if the entries in each row of $U$ increase from left to right and if the entries in each column increase from top to bottom. 

\begin{thm}[Standard basis theorem]\label{stb} A basis of the Specht module  $\tilde{M}^{\lambda} / Q^{\lambda}$ is formed by the cosets $q+Q^{\lambda}$, where $q$ varies over all column tabloids that correspond to standard tableaux of shape $\lambda$.
\end{thm}
\begin{rem}\label{Zlin}
	From the proof of \cite[Lemma 6, Section 7.4]{F} it follows that if $q+Q^{\lambda} \in \tilde{M}^{\lambda} / Q^{\lambda}$, where $q$ is a column tabloid, then $q+Q^{\lambda}$ may be written as a $\mathbb{Z}$-linear combination of the basis elements of the previous theorem.
\end{rem}
Another well known result we will need is that if $\lambda=(\lambda_1, \lambda_2)$ is a partition of $n$, then the dimension of the Specht module $S^{\lambda}$ is equal to \begin{equation}\label{dimsp}\dim(S^{\lambda})=\tbinom{n}{\lambda_2} -  \tbinom{n}{\lambda_{2} -1}.
	\end{equation} 
	
	The irreducible decomposition of the permutation module $M_k$ is given by Young's rule \cite[Corollary 1, Section 7.3]{F},\begin{equation}\label{Young}M_k=S^{(n)} \oplus S^{(n-1,1)} \oplus \cdots \oplus S^{(n-s,s)}, \ s=\min\{k,n-k\}.\end{equation}
\section{Total trades}
In this section we prove that the space generated by the set of $t$-$(n,k)$ total trades is an irreducible module of the symmetric group $\mathfrak{S}_n.$
\subsection{An identity} In this subsection we prove an identity that will be used in th proof of Theorem \ref{main0}.
\begin{defn}\label{def1}For a nonnegative integer $m$ and a subset $A$ of $X$ let \begin{equation}
		\Sigma_m(A):=\sum x_{i_1}\cdots x_{i_m} \in M_m
	\end{equation}
	where the sum is over all $m$-subsets $\{x_{i_1},\dots, x_{i_m} \}$ of $A$. 
	\end{defn}
	In the above definition, our convention is that for $m=0$ we have $\Sigma_0(A)=1$ and for $m>|A|$ we have $\Sigma_m(A)=0.$

\begin{lem}\label{1}Suppose $B$ is a subset of $X$ and $w$ is an element in $X$ such that $w \notin B$. Then for $m\ge1$ \[\Sigma_m(X\setminus B) = \Sigma_m(X\setminus(B\cup \{w\})) + w\Sigma_{m-1}(X\setminus(B\cup \{w\})).\]
\end{lem}
\begin{proof}
	This follows immediately from Definition \ref{def1}.
\end{proof}

We will need the identity described in the next lemma.
\begin{lem}\label{id}Suppose $A$ is a subset of $X$ and $x,y,z$ are distinct elements of $X$ such that $x,y,z \notin A$. Then for $m \ge 0$ we have the following identity in $M_m$,
	\begin{equation}\label{sigmaid}(x-y)\Sigma_m(X \setminus (A \cup \{x,y\}))-(z-y)\Sigma_m(X \setminus (A \cup \{z,y\}))-(x-z)\Sigma_m(X \setminus (A \cup \{x,z\}))=0.\end{equation}
\end{lem}
\begin{proof} The identity is clear for $m=0$. Suppose $m \ge 1$. Let $L$ be the left hand side of (\ref{sigmaid}) and let $C:=A\cup\{x,y,z\}$. From Lemma \ref{1} we have \begin{align*}
		\Sigma_m(X \setminus (A \cup \{x,y\}))=\Sigma_m(X \setminus C)+z\Sigma_{m-1}(X \setminus C),\\
		\Sigma_m(X \setminus (A \cup \{z,y\}))=\Sigma_m(X \setminus C)+x\Sigma_{m-1}(X \setminus C),\\
		\Sigma_m(X \setminus (A \cup \{x,z\}))=\Sigma_m(X \setminus C)+y\Sigma_{m-1}(X \setminus C).
	\end{align*}
Substituting in $L$ we see that  	
\[L=c_0\Sigma_m(X \setminus C) + c_1\Sigma_{m-1}(X \setminus C),	\]
where 
	$c_0=(x-y)-(z-y)-(x-z)$ and  
		$c_1=(x-y)z-(z-y)x-(x-z)y.$
	We have $c_0=c_1=0$ and thus $L=0$. \end{proof}

\subsection{Main result}
The main result of the present section is the following.

\begin{thm}\label{main0}
	Let $t< k $ and $t+k \le n$. Then as $\mathfrak{S}_n$-modules the  subspace $\langle \mathfrak{T}_{t,k,n}\rangle$ of $M_k$ spanned by the set $\mathfrak{T}_{t,k,n}$ of $t$-$(n,k)$ total trades is isomorphic to the Specht module $S^{(n-t-1,t+1)}$.
\end{thm} 
\begin{proof} We define a map  $h:\tilde{M}^{(n-t-1,t+1)} \to \langle \mathfrak{T}_{t,k,n}\rangle$ sending the column tabloid
\begin{center}
	$q:=$	\begin{tabular}{ |c|c|c| c|c|c| } 
			$x_1$ & $\cdots$ & $x_{t+1}$ & $z_{t+2}$ & $\cdots$ & $z_{n-t-1}$\\ 
			$y_1$ & $\cdots$ & $y_{t+1}$ \\ 
			\end{tabular}\end{center}
to the $t$-$(n,k)$ total trade 	
\[T:=(x_1-y_1)\cdots (x_{t+1}-y_{t+1})\sum x_{i_{t+2}}\cdots x_{i_k},\] where the sum is over all $(k-t-1)$-subsets $\{x_{i_{t+2}}, \dots, x_{i_k}\}$ of $X \setminus \{x_1,y_1,\dots, x_{t+1}, y_{t+1}\}$.	Since $T$ is skew symmetric with respect to $x_i,y_i$, where $i=1,\dots,t+1$, the map $h$ is well defined. Moreover it easily follows that $h$ is a surjective map of $\mathfrak{S}_n$-modules.

We will show that $h(Q^{(n-t-1,t+1)})=0$ where $Q^{(n-t-1,t+1)}$ is the subspace of $\tilde{M}^{(n-t-1,t+1)}$ given in Definition \ref{Q}.

As we saw in Section 2.1, every generator $g_{U,c}$ of $Q^{(n-t-1,t+1)}$ depends on a pair of consecutive columns $(c,c+1)$ of $q=q_{U}$. We distinguish three cases.

Case 1. Consider column $c$ of $q$, where $c \le t$. Let $q_1,q_2$ be the following column tabloids
\begin{align*} q_1&:=\begin{tabular}{ |c|c|c|c|c|c|c|c|c| } 
		$x_1$ & $\cdots$ & $x_{\bf{c+1}}$ & $x_{\bf{c}}$ & $\cdots$  & $x_{t+1}$ & $z_{t+2}$ & $\cdots$ & $z_{n-t-1}$\\ 
		$y_1$ & $\cdots$ & $y_{c}$ & $y_{c+1}$ & $\cdots$  & $y_{t+1}$ \\
\end{tabular}\ , \\
q_2&:=\begin{tabular}{ |c|c|c|c|c|c|c|c|c| } 
	$x_1$ & $\cdots$ & $x_{c}$ & $y_{\bf{c}}$ & $\cdots$  & $x_{t+1}$ & $z_{t+2}$ & $\cdots$ & $z_{n-t-1}$\\ 
	$y_1$ & $\cdots$ & $x_{\bf{c+1}}$ & $y_{c+1}$ & $\cdots$  & $y_{t+1}$ \\ 
\end{tabular}\ ,
\end{align*}
where $q_i$ is obtained by exchanging the top element of column $c+1$ of $q$ with the $i$th element of column $c$ of $q$, where $i=1,2$. The indices of the exchanged elements have been denoted with boldface characters. From the definition of $g_{U,c}$ we have $g_{U,c}=q-q_1-q_2$ and from the definition of $h$ we have $h(q-q_1-q_2)=T-T_1-T_2,$ where 
\begin{align*}
T_1&:=(x_1-y_1)\cdots (x_{c+1} - y_c)(x_c -y_{c+1}) \cdots (x_{t+1}-y_{t+1})\sum x_{i_{t+2}}\cdots x_{i_k}\\
T_2&:=(x_1-y_1)\cdots (x_c-x_{c+1})(y_c-y_{c+1}) \cdots (x_{t+1}-y_{t+1})\sum x_{i_{t+2}}\cdots x_{i_k},
\end{align*}
and both sums are as the sum in $T$. A quick computation yields \[ (x_c-y_c)(x_{c+1}-y_{c+1})-(x_{c+1} - y_c)(x_c -y_{c+1}) - (x_{c} - x_{c+1})(y_c -y_{c+1})=0.\]
Hence $T-T_1-T_2=0$. Thus we have $h(g_{U,c})=0$ as desired.

Case 2. Consider column $c=t+1$ of $q$. This is similar to the previous case but there is one extra step. Let $q_1,q_2$ be the following column tabloids
\begin{align*} q_1&:=\begin{tabular}{ |c|c|c|c|c|c|c| } 
		$x_1$ & $\cdots$  & $x_{{t}}$   & $z_{\bf{t+2}}$ & $x_{\bf{t+1}}$ & $\cdots$ & $z_{n-t-1}$\\ 
		$y_1$ & $\cdots$  & $y_{t}$ & $y_{t+1}$ \\
	\end{tabular}\ , \\
	 q_2&:=\begin{tabular}{ |c|c|c|c|c|c|c| } 
		$x_1$ & $\cdots$  & $x_{{t}}$   & $x_{t+1}$ & $y_{\bf{t+1}}$ & $\cdots$ & $z_{n-t-1}$\\ 
		$y_1$ & $\cdots$  & $y_{t}$ & $z_{\bf{t+2}}$ \\
	\end{tabular}\ ,
\end{align*}
where $q_i$ is obtained by exchanging the element $z_{{t+2}}$ of $q$ with the $i$th element of column $t+1$ of $q$, where $i=1,2$. 
We have $g_{U,c}=q-q_1-q_2$ and  $h(q-q_1-q_2)=T-T_1-T_2,$ where 
\begin{align*}
	T_1&:=(x_1-y_1)\cdots (x_{t} - y_t)(z_{t+2} -y_{t+1}) \sum x_{j_{t+2}}\cdots x_{j_k}\\
	T_2&:=(x_1-y_1)\cdots (x_{t} - y_t)(x_{t+1} -z_{t+2}) \sum x_{l_{t+2}}\cdots x_{l_k},
\end{align*} and the sums are as follows.
\begin{itemize}
	\item The sum  $S(1):=\sum x_{j_{t+2}}\cdots x_{j_k}$ in $T_1$ is over all $(k-t-1)$-subsets $\{x_{j_{t+2}}, \dots, x_{j_k}\}$ of $X \setminus \{x_1,y_1,\dots, x_{t}, y_{t}, z_{t+2}, y_{t+1}\}$.  
	\item the sum $S(2) :=\sum x_{j_{t+2}}\cdots x_{j_k}$ in $T_2$ is over all $(k-t-1)$-subsets $\{x_{l_{t+2}}, \dots, x_{l_k}\}$ of $X \setminus \{x_1,y_1,\dots, x_{t}, y_{t}, x_{t+1}, z_{t+2}\}$.
	\end{itemize}
Let $S:=\sum x_{i_{t+2}}\cdots x_{i_k}$ be the sum in $T$. So this sum is over all $(k-t-1)$-subsets $\{x_{i_{t+2}}, \dots, x_{i_k}\}$ of $X \setminus \{x_1,y_1,\dots, x_{t+1}, y_{t+1}\}$. With the notation of Lemma \ref{id} we have \begin{align*}S&=\Sigma_{k-t-1}(X  \setminus (A \cup \{x_{t+1},y_{t+1}\})\\
S(1)&=\Sigma_{k-t-1}(X  \setminus (A \cup \{z_{t+2},y_{t+1}\})\\
S(2)&=\Sigma_{k-t-1}(X  \setminus (A \cup \{x_{t+1},z_{t+2}\})
\end{align*}
where $A=\{x_1,y_1, \dots, x_t,y_t\}$. By Lemma \ref{id} we have \[ (x_{t+1}-y_{t+1})S-(z_{t+2} - y_{t+1})S(1) - (x_{t+1} - z_{t+2})S(2)=0.\]
Hence $T-T_1-T_2=0$. Thus $h(g_{U,c})=0$ as desired.

Case 3. Consider column $c$ of $q$, where $c \ge t+2.$ Let $q_1$ be the column tabloid obtained from $q$ by exchanging the elements $z_c$ and $z_{c+1}$. Then $g_{U,c}=q-q_1$. Since $c \ge t+2$, it follows from the definition of the total trade $T$ in the first paragraph of the proof, that $h(q)=h(q_1)=T$. Hence $h(g_{U,c})=0$ as desired.

From the three cases considered above, it follows that every generator $g_{U,c}$ of the module $Q^{(n-t-1,t+1)}$ is mapped to zero under $h:\tilde{M}^{(n-t-1,t+1)} \to \langle \mathfrak{T}_{t,k,n}\rangle$. Thus $h(Q^{(n-t-1,t+1)})=0$. From Theorem \ref{pres} we have that the map $h$ induces a map of $\mathfrak{S}_n$-modules \[\bar{h}:S^{(n-t-1,t+1)} \to \langle \mathfrak{T}_{t,k,n}\rangle.\] Since the map $h$ is surjective, the map $\bar{h}$ is surjective. Now since the $\mathfrak{S}_n$-module $S^{(n-t-1,t+1)}$ is irreducible and $\langle \mathfrak{T}_{t,k,n}\rangle$ is nonzero , the map $\bar{h}$ is an isomorphism.\end{proof}
\subsection{Corollaries} We will consider some corollaries of Theorem \ref{main0} and discuss  briefly how these relate to earlier work on trades.

There is interest in obtaining bases of the $\mathbb{Z}$-module spanned by all $t$-$(n,k)$ trades, for example see \cite{GJ,GLL,KM,KTR}. The next result gives a basis of the $\mathbb{Z}$-module spanned by the the $t$-$(n,k)$ total trades.
  \begin{cor}\label{basis}
 	Let $t< k $ and $t+k \le n$. Let us fix a total order on the set $X$. Then a basis of the vector space $\langle \mathfrak{T}_{t,k,n}\rangle$ is the set of $t$-$(n,k)$ total trades \[T=(x_1-y_1)\cdots (x_{t+1}-y_{t+1})\sum x_{i_{t+2}}\cdots x_{i_k} \in M_{k}\] that satisfy the conditions
 	\begin{itemize}
 		\item $x_1 < x_2 < \cdots < x_{t+1}$,
 		\item $y_1 < y_2 < \cdots < y_{t+1}$, and
 		\item $x_i < y_i$, $i=1,\dots, t+1$.
 \end{itemize}
Moreover, the above $t$-$(n,k)$ total trades form a basis of the $\mathbb{Z}$-module spanned by $\mathfrak{T}_{t,k,n}$. \end{cor} 
 \begin{proof}We saw in the proof of Theorem \ref{main0} that the map  $h:\tilde{M}^{(n-t-1,t+1)} \to \langle \mathfrak{T}_{t,k,n}\rangle$ sending the column tabloid
 	\begin{center}
 			\begin{tabular}{ |c|c|c| c|c|c| } 
 			$x_1$ & $\cdots$ & $x_{t+1}$ & $z_{t+2}$ & $\cdots$ & $z_{n-t-1}$\\ 
 			$y_1$ & $\cdots$ & $y_{t+1}$ \\ 
 	\end{tabular}\end{center}
 	to the $t$-$(n,k)$ total trade $T$ induces an isomorphism $\tilde{M}^{(n-t-1,t+1)}/Q^{(n-t-1,t+1)} \to \langle \mathfrak{T}_{t,k,n}\rangle$. The desired result follows from Theorem \ref{main0}.
 	
 	The second claim of the corollary follows from Remark \ref{Zlin}.\end{proof} We remark that the basis elements of the previous corollary resemble Specht polynomials \cite{Pe}, except for the `tail part' $\sum x_{i_{t+2}}\cdots x_{i_k}$ of $T$. 

It was shown in \cite[Theorem 11]{GKK} that the vector space of $t$-$(n,k)$ trades decomposes as the direct sum $\langle \mathfrak{T}_{t,k,n}\rangle \oplus \cdots \oplus \langle \mathfrak{T}_{k-1,k,n}\rangle$ of subspaces of total trades. We recover this result from Theorem \ref{main0} as the irreducible decomposition of the space of $t$-$(n,k)$ trades, i.e. as the irreducible decomposition of the kernel of the inclusion matrix $W_{t,k}$.
\begin{cor}[\cite{GKK}]\label{cor2} Let $t< k \le n/2$. Then vector space of $t$-$(n,k)$ trades decomposes as the direct sum $\langle \mathfrak{T}_{t,k,n}\rangle \oplus \cdots \oplus \langle \mathfrak{T}_{k-1,k,n}\rangle$.\footnote{In \cite[Theorem 11]{GKK} the result is stated without assumptions on $k,n$. However, in its proof \cite[Theorem 9]{GKK} is used which has such assumptions. These yield $2k-1 \le n$.}
\end{cor}
\begin{proof}
We know that the matrix of the linear map $\psi_k^{(k-t)}: M_k \to M_t$ with respect to some ordering of the basis of $k$-subsets and $t$-subsets of $X$ is the inclusion matrix $W_{t,k}$. From Young's rule (\ref{Young}), the irreducible decompositions of the $\mathfrak{S}_k$-modules $M_k$ and $M_t$ are 	$M_k = S^{(n-k,k)} \oplus S^{(n-k+1,k-1)} \oplus \cdots \oplus S^{(n)}$ and 
$M_t = S^{(n-t,t)} \oplus S^{(n-t+1,t-1)} \oplus \cdots \oplus S^{(n)}$. Since the matrix $W_{t,k}$ has maximal rank \cite{Go}, it follows that the irreducible decomposition of the kernel of $\psi_k^{(k-t)}$ is \begin{equation}\label{irrker}\ker (\psi_k^{(k-t)}) = S^{(n-k,k)} \oplus S^{(n-k+1,k-1)} \oplus \cdots \oplus S^{(n-t-1,t+1)}.\end{equation} From Theorem \ref{main0} we have $S^{(n-i-1,i+1)} \cong  \langle \mathfrak{T}_{i,k,n}\rangle$ for all $i=t, \dots, k-1$. (The hypothesis $k \le n/2$ is needed in order to apply Theorem \ref{main0} for $t$ replaced by $k-1$. Also it is needed so that $(n-k,k)$ is a partition.)\end{proof}
 For a $t$-$(n,k)$ trade $T$ we denote by $\Pi(T)$ the orbit of $T$ under the action of the symmetric group $\mathfrak{S}_n$. The following decomposition result was proved in \cite[Theorem 2]{GKK}.
\begin{cor}[{\cite{GKK}}]Let $t< k \le n/2$ and let $T$ be a $t$-$(n,k)$ trade.Then \[\langle \Pi(T) \rangle =\bigoplus_{i\in I}\langle \mathfrak{T}_{i,k,n}\rangle,\] where $I \subseteq \{t,\dots, k-1\}$ consists of integers $i$ such that there is an $i$-trade in $\langle \Pi(T) \rangle$ which is not an $(i+1)$-trade. In particular, $\langle \Pi(T) \rangle$ is the whole vector space of $t$-$(n,k)$ trades if and only if $I=\{t,\dots, k-1\}$.\end{cor}
\begin{proof} From Corollary \ref{cor2} it follows that the irreducible decomposition of the $\mathfrak{S}_n$-module $\langle \Pi(T) \rangle$ is of the form $\bigoplus_{i\in I}\langle \mathfrak{T}_{i,k,n}\rangle$, where $I$ is a nonempty subset of $\{t, \dots, k-1\}.$ For every $i \in I$, the set $\mathfrak{T}_{i,k,n}$ consists of $i$-$(n,k)$ trades. These are not $(i+1)$-$(n,k)$ trades; indeed,  by Theorem \ref{main0} we have $\langle \mathfrak{T}_{i,k,n}\rangle =S^{(n-i-1,i+1)}$ and this Specht module is not contained in the kernel of the map $\psi_k^{(k-i-1)} :M_{k} \to M_{i+1}$ according to (\ref{irrker}) for $i+1$ in place of $t$.\end{proof}
A classical result of Granert and Jurkat states that the $\mathbb{Z}$-module spanned by all $t$-$(n,k)$ trades is spanned by the permutations of any nonzero minimal $t$-$(n,k)$ trade. As an immediate corollary of Theorem \ref{main0} we recover the vector space version of this result.
\begin{cor}[\cite{GJ}] Let $t< k \le n/2$. The vector space of $t$-$(n,k)$ trades is spanned by the permutations of any nonzero minimal $t$-$(n,k)$ trade.
\end{cor}
\begin{proof} Since every total trade is a sum of minimal $t$-$(n,k)$ trades, we see from Corollary \ref{cor2} that the vector space of $t$-$(n,k)$ trades is spanned by the minimal $t$-$(n,k)$ trades. From this and the observation that the symmetric group $\mathfrak{S}_n$ acts transitively on minimal trades the result follows. \end{proof}	

\section{Representations of the general linear group}
The purpose of this section is to recall basis results from the polynomial representation theory of the general linear group that will be needed in the next section. Our main reference here is the book by Weyman \cite{W}.
\subsection{Divided power algebra}
Let $G=GL_N(\mathbb{K})$, where $N n \ge 2$. We denote by $V=\mathbb{K}^N$ be the natural $G$-module  of column vectors.

By $D=\bigoplus_{i\geq 0}D_i$ we denote the divided power algebra of $n$. We will recall some definitions and facts concerning this algebra. For more details we refer to \cite[Section 1.1]{W}. 

We recall that $D$ is defined as the graded dual of the symmetric algebra $S(V^*)$ of $V^*$, where $V^*$ is the dual of $V$. So by definition we have $D_i = (S_i(V^*))^*.$ 

Since the characteristic of $\mathbb{K}$ is zero, $D$ is naturally isomorphic to the symmetric algebra $V$. However, the computations to be made in Sections 5 seem somewhat less involved if one deals with Weyl modules in place of Schur modules. For this reason we work with the divided power algebra and Weyl modules.

If $v \in V$ and $i$ is a nonnegative integer, we have the $i$th divided power  $v^{(i)} \in D_i$ of $v$. In particular, $ v^{(0)}=1 \ \mathrm{and} \ v^{(1)}=v$ for all $v \in V$. We recall that if $i,j$ are nonnegative integers, then the product $v^{(i)}v^{(j)}$ of $v^{(i)}$ and $v^{(j)}$ is given by $v^{(i)}v^{(j)}=\tbinom{i+j}{j}v^{(i+j)},$ where $\tbinom{i+j}{j}$ is the indicated binomial coefficient. 

If $\{e_1, \dots, e_N\}$ is a basis of the vector space $V$, then a basis of the vector space $D_i$ is the set $\{e_1^{(\alpha_1)}\cdots e_N^{(\alpha_N)}: \alpha_1+\cdots + \alpha_N = i\}.$

We recall that $D$ has a graded Hopf algebra structure. Let $\Delta : D \to D \otimes D$ be the comultiplication map of $D$. Explicitly, for a homogeneous element $x= v_1^{(\alpha_1)}\cdots v_t^{(\alpha_t)} \in D_a$, where $v_i \in V$, we have \[ \Delta(x)=\sum_{0\le \beta_i \le \alpha_i} v_1^{(\beta_1)}\cdots v_t^{(\beta_t)} \otimes v_1^{(\alpha_1 - \beta_1)}\cdots v_t^{(\alpha_t - \beta_t)}.\]
For $0 \le b \le a$ we may restrict the above sum to those $\beta_i$ such that $\beta_1 + \cdots +  \beta_t= b$. This yields the following component of the comultiplication map \begin{align*} D_a &\to D_b \otimes D_{a-b}, \\  x &\mapsto \sum_{\substack{0\le \beta_i \le \alpha_i \\ \beta_1+\cdots+\beta_t=b}} v_1^{(\beta_1)}\cdots v_t^{(\beta_t)} \otimes v_1^{(\alpha_1 - \beta_1)}\cdots v_t^{(\alpha_t - \beta_t)}, \end{align*}
which we denote by $\Delta_{b,a-b}:  D_a \to D_b \otimes D_{a-b}$.

\subsection{Weyl modules and Pieri's rule}

If $A, B$ are $G$-modules, then the tensor product $A \otimes B$ becomes a $G$-module with the diagonal action of $G$ given by $g(a \otimes b) = ga \otimes gb$, where $g \in G$, $a \in A, b \in B$. As in Section 2 we will restrict our attention to partitions with at most two parts. For a partition $\mu=(\mu_1, \mu_2)$ we denote by $K_\mu$ the corresponding Weyl module for $G$ defined in \cite[Section 2.1]{W}.  We recall that there is a surjective map of $G$ modules \[\pi_\mu: D_{\mu_1} \otimes D_{\mu_2} \to K_\mu\] defined uniquely up to nonzero scalar multiple (cf. \cite[p. 43]{W} where the notation $\psi_{\lambda / \mu}$ is used for the more general case of skew partitions $\lambda / \mu$ in place of partitions). For example, when $\mu=(a)$ consists of one part, then $K_{(a)}=D_a$. 

Since the characteristic of $\mathbb{K}$ is zero, $K_\mu$ is an irreducible $G$-module. Moreover, if $\lambda$ and $\mu$ are distinct partitions, then the irreducible modules $K_\lambda$ and $K_\mu$ are not isomorphic.

 Suppose $a \ge b$ are nonnegative integers. Then the irreducible decomposition of the $G$-module $D_a \otimes D_b$ is given by Pieri's rule \cite[(2.3.5) Corollary]{W}
\begin{equation}
	D_a \otimes D_b = \bigoplus_{j=0}^{b} K_{(a+b-j,j)}.
\end{equation}
For each $j=0, \dots b$, there is a unique (up to nonzero scalar multiple) surjective map of $G$-modules $\pi_j: D_a\otimes D_a \to K(a+b-j,j)$ which is the composition \begin{equation}\label{pij} \pi_j: 
D_a \otimes D_b \xrightarrow{\Delta_{a+b-j,j}} D_{a+b-j} \otimes D_{j} \xrightarrow{\pi_{(a+b-j,j)}} K(a+b-j,j).
\end{equation}

\subsection{Straightening row semistandard tableaux} 
A Young tableau with entries from $\{1,2, \dots, N \}$ is called semistandard if the elements weakly increase in the rows from left to right and strictly increase in the columns from top to bottom. In the sequel we will need to express elements of Weyl modules as explicit linear combinations of semistandard basis elements.  The next lemma concerns violations of semistandardness in the first column. In order to have simplified notation, we will denote the basis elements $e_1, \dots, e_N$ by their subscripts $1, \dots, N$, respectively. Thus, for example, the element $e_1^{(2)}e_3^{(5)}$ of $D_7$ will be depicted as $1^{(2)}3^{(5)}$.
\begin{lem}[{\cite[Lemma 4.2]{MS3}}]\label{lemglas}Let $\mu=(\mu_1,\mu_2)$ be a partition of length two and let $Z:=
		{1}^{(a_1)}{2}^{(a_2)}  \cdots   {N}^{(a_N)} \otimes
		{1}^{(b_1)}{2}^{(b_2)}  \cdots  {N}^{(b_N)}
	 \in D_{\mu_1} \otimes D_{\mu_2}.$ 
	Then we have the following identities in $K_{\mu}$.
	\begin{enumerate}
		\item If $a_1+b_1>\mu_1$, then $\pi_\mu(Z)=0$.
		\item If $a_1+b_1 \le \mu_1$, then 
		\begin{equation}\label{eqglas}\pi_\mu(Z)=(-1)^{b_1}\sum_{k_2,\dots,k_n}\tbinom{b_2+k_2}{b_2}\cdots\tbinom{b_N+k_N}{b_N}
			\pi_\mu(Z(k_2, \dots, k_N)),
		\end{equation} where \[Z(k_2,\dots,k_N):= 
			{1}^{(a_1+a_2)}{2}^{(a_2-k_2)}  \cdots   {N}^{(a_N-k_N)} \otimes
			{2}^{(b_2+k_2)}  \cdots  {N}^{(b_N+k_N)}
		\] and the sum ranges over all  nonnegative integers $k_2,\dots,k_N$
		such that  $k_2+\dots+k_N=b_1 $ and $k_s \le a_s$ for all $s=2,\dots,N$.	\end{enumerate}	
\end{lem}
Even though our paper \cite{MS3} concerns modular representations, the proof of the above lemma given there is valid for any field in place of $\mathbb{K}$ (in fact for any commutative ring). In \cite[Lemma 4.2]{MS3} we used the notation $\Delta_\mu$ for the Weyl module $K_\mu$.

Given $Z:=
{1}^{(a_1)}{2}^{(a_2)}  \cdots   {N}^{(a_N)} \otimes
{1}^{(b_1)}{2}^{(b_2)}  \cdots  {N}^{(b_N)}
\in D_{\mu_1} \otimes D_{\mu_2}$, let $T(Z)$ be the Young tableau with first row consisting (from left to right) of $a_1$ $1$'s, \dots, $a_N$ $N$'s and second row consisting (from left to right) of $b_1$ $1$'s, \dots, $b_N$ $N$'s. We will need the following theorem.\begin{thm}[{\cite[(2.1.15) Proposition]{W}}]A basis of the vector space $K_{\mu}$ is the set $ \{ \pi_\mu(Z): T(Z) \ \text{\ is \ a \ standard \ Young \ tableau} \}.$ \end{thm}

\section{Images of intersection matrices}
Throughout this section we fix nonnegative integers $ t, k, n $ satisfying \begin{equation}\label{int}  t \le k \le n/2. \end{equation}
We recall that our convention on binomial coefficients is $\tbinom{a}{b}=0$ if $b>a$ or $b<0$. The following integers are important for Theorem \ref{main1}.
\begin{defn}\label{lambda} Suppose $l$ is a nonnegative integer such that $l \le t$. For $j=0,1, \dots, t$ define the integer \[\lambda_j(t,k,n;l) :=\sum_{s=0}^{j}(-1)^{j-s}\tbinom{j}{s}\tbinom{k-s}{l-s}\tbinom{n-k-j+s}{t-l-j+s}.\]
\end{defn}

\begin{lem}\label{bas}A basis of the vector space $\Hom_G(D_{n-k}\otimes D_k, D_{n-t}\otimes D_t)$ is the set $\{ \Psi_0, \dots, \Psi_t \}$, where the map $\Psi_l$ is the composition 
	\begin{align*}D_{n-k} \otimes D_k &\xrightarrow{\Delta_{n-k-t+l,t-l} \otimes \Delta_{k-l,l}}D_{n-k-t+l} \otimes D_{t-l}\otimes D_{k-l} \otimes D_l \\&\xrightarrow{1 \otimes \tau \otimes 1} D_{n-k-t+l} \otimes D_{k-l}\otimes D_{t-l} \otimes D_l \\&\xrightarrow{m_{n-k-t+l,k-l} \otimes m_{t-l,l}}D_{n-t}\otimes D_t,\end{align*}
	where $\tau :D_{t-l} \otimes D_{k-l} \to D_{k-l} \otimes D_{t-l}$ is given by $\tau(x\otimes y) = y \otimes x$.
\end{lem}
	\begin{proof} From Pieri's rule we have $D_{n-k} \otimes D_{k} = \bigoplus_{j=0}^{k} K_{(n-j,j)}$ and $D_{n-t} \otimes D_{t} = \bigoplus_{j=0}^{t} K_{(n-j,j)}$. Since $t \le k$ we conclude from Schur's lemma that the dimension  of $\Hom_G(D_{n-k}\otimes D_k, D_{n-t}\otimes D_t)$ is equal to $t$.
		
		From the definition of the map $\Psi_l$ we have \[\Psi_l(1^{(n-k)}\otimes 2^{{(k)}})=1^{(n-k-t+l)}2^{(k-l)}\otimes 1^{(t-l)} 2^{(l)}.\] Hence the maps $\Psi_0, \dots, \Psi_t$ are linearly independent and thus are a basis of the vector space  $\Hom_G(D_{n-k}\otimes D_k, D_{n-t}\otimes D_t)$.
		\end{proof}
This first main result of the present section is the following.	\begin{thm}\label{main1}
		Suppose $\Psi \in \Hom_G(D_{n-k}\otimes D_k, D_{n-t}\otimes D_t)$ and let $\Psi = \sum_{l=0}^{t} c_l \Psi_l$ where $c_l \in \mathbb{K}$. Then the irreducible decomposition of the image of the map $\Psi$ is 
		\begin{equation}\label{irrdec1} \Ima(\Psi) = \bigoplus_{j \in J(t,k,n)}K_{(n-j,j)},\end{equation} where $J(t,k,n)$ is the subset of $\{0,1, \dots, t\}$ consisting of those $j$ such that there exists an $l \in \{0,1,\dots,t\}$ satisfying $c_l \neq 0$ and $\lambda_j(t,k,n;l) \neq 0.$
	\end{thm}
\begin{proof}
First, we compute the composition \begin{equation}\label{comp1}D_{n-k} \otimes D_k \xrightarrow{\Psi_l} D_{n-t} \otimes D_t \xrightarrow{\pi_j} K_{(n-j,j)}.\end{equation} Using the definitions of $\Psi_l$ and $\pi_j$ (for the latter see (\ref{pij})) we have \begin{equation}
	\Psi_l(1^{(n-k)}\otimes2^{(k)})=1^{(n-k+l-t)}2^{(k-l)} \otimes 1^{(t-l)} 2^{(l)}
\end{equation}
and 
\begin{equation}\label{im33}
	\pi_j \circ \Psi_l(1^{(n-k)}\otimes2^{(k)})=\sum_{j_1+j_2=j}\tbinom{n-k-j_1}{t-l-j_1}\tbinom{k-j_2}{l-j_2} \pi_{(n-j,j)}(1^{(n-k-j_1)}2^{(k-j_2)} \otimes 1^{(j_1)} 2^{(j_2)}),
\end{equation}
where the sum is over all nonnegative integers $j_1, j_2$ satisfying \begin{align*}j_1+j_2&=j, \\  j_1 &\le t-l, \\ j_2 &\le l.\end{align*}  We have $j_1 =j-j_2 \le t-j_2 \le k -j_2$. Hence we may apply Lemma \ref{lemglas}(2) to each term $\pi_{(n-j,j)}(1^{(n-k-j_1)}2^{(k-j_2)} \otimes 1^{(j_1)} 2^{(j_2)})$ of the right hand side of eq. (\ref{im33}) to obtain
\begin{align}\label{im34}
	&\pi_{(n-j,j)}(1^{(n-k-j_1)}2^{(k-j_2)} \otimes 1^{(j_1)} 2^{(j_2)})\\\nonumber&=(-1)^{j_1}\tbinom{j_1+j_2}{j_2}\pi_{(n-j,j)}(1^{(n-k)}2^{(k-j_2-j_1)} \otimes 2^{(j_1+j_2)} )\\\nonumber&=(-1)^{j-j_2}\tbinom{j}{j_2}\pi_{(n-j,j)}(1^{(n-k)}2^{(k-j)}\otimes 2^{(j)}).\end{align}
Substituting eq. (\ref{im34}) in eq. (\ref{im33}) we obtain
\begin{align}\label{im35}
	&\pi_j \circ \Psi_l(1^{(n-k)}\otimes2^{(k)})\\\nonumber&=\Big(\sum_{j_1+j_2=j}(-1)^{j-j_2}\tbinom{n-k-j_1}{t-l-j_1}\tbinom{k-j_2}{l-j_2} \tbinom{j}{j_2}\Big) \pi_{(n-j,j)}(1^{(n-k)}2^{(k-j)}\otimes 2^{(j)})\\\nonumber&= \Big( \sum_{j_2}(-1)^{j-j_2}\tbinom{n-k-j+j_2}{t-l-j+j_2}\tbinom{k-j_2}{l-j_2} \tbinom{j}{j_2}\Big)\pi_{(n-j,j)}(1^{(n-k)}2^{(k-j)}\otimes 2^{(j)}),\end{align} where the second sum is over all nonnegative integers $j_2$ satisfying \[\max\{j-t+l\} \le j_2 \le \min\{l,j\}.\]
According to our convention on binomial coefficients we have $\tbinom{a}{b}=0$ if $b<0$ or $b>a$. Hence we may write eq. (\ref{im35}) as
\begin{align*}\label{im36}
	&\pi_j \circ \Psi_l(1^{(n-k)}\otimes2^{(k)})\\&= \Big( \sum_{j_2 = 0}^{j} (-1)^{j-j_2}\tbinom{n-k-j+j_2}{t-l-j+j_2}\tbinom{k-j_2}{l-j_2} \tbinom{j}{j_2}\Big) \pi_{(n-j,j)}(1^{(n-k)}2^{(k-j)}\otimes 2^{(j)})\end{align*}
and thus using Definition \ref{lambda} we obtain 
\begin{equation}\label{im36}
	\pi_j \circ \Psi_l(1^{(n-k)}\otimes2^{(k)})= \lambda_j(t,k,n;l) \pi_{(n-j,j)}(1^{(n-k)}2^{(k-j)}\otimes 2^{(j)}).\end{equation}

Now from (\ref{int}) we have $t \le n-k$ and thus $j \le n-k$. This means that the element $\pi_{(n-j,j)}(1^{(n-k)}2^{(k-j)}\otimes 2^{(j)})$ of the Weyl module $K_{(n-j,j)}$ in the right hand side of eq. (\ref{im36}),  is nonzero since the corresponding tableau is semistandard. Hence the map (\ref{comp1}) is nonzero if and only if $\lambda_j(t,k,n;l) \neq 0$. Equivalently, the irreducible module $K_{(n-j,j)}$ is a summand of the image $\Ima (\Psi_l)$ of $\Psi_l$ if and only if $\lambda_j(t,k,n;l) \neq 0$. From this, the fact that $\Psi = \sum_{l=0}^{t} c_l\Psi_l$, and the irreducible decomposition $D_{n-t} \otimes D_t =\bigoplus_{j=0}^{t} K_{(n-j,j)}$ of the codomain $D_{n-t} \otimes D_t$ of the map $\Psi : D_{n-k} \otimes D_k \to D_{n-t} \otimes D_t$ according to Pieri's rule, we conclude that the irreducible decomposition of $\Ima(\Psi)$ is given by (\ref{irrdec1}).\end{proof}

\begin{rem} We note that integers closely related to  $\lambda_j(t,k,n;l)$ have appeared in \cite[Corollary 4.4]{MMS} in the problem of determining presentations of Schur modules and Specht modules. In that paper, $\Hom_G(\Lambda^{n-k}\otimes \Lambda^k, \Lambda^{n-k}\otimes \Lambda^k)$
was studied, where $\Lambda^{j}$ is the $i$th exterior power of the natural $G$-module $V$. Note that there we have $k=t$ so that in the present paper we are considering a more general situation. Also the  $\lambda_j(t,k,n;l)$ have appeared in \cite[Definition 1.4]{DEGJPP} where diagonal forms over the integers of the intersection matrices $U_{t,k,n}$ were studied.\end{rem}

We now consider the Schur functor. Suppose $N \ge n$. We recall from \cite[Section 6.3]{Gr} that the Schur functor is a functor $f$ from the category of homogeneous polynomial representations of  $G$ of degree $n$ to the category of $\mathfrak{S}_n$-modules. For $M$ an object in the first category,  $f(M)$ is the weight subspace $M_\alpha$ of $M$, where $\alpha = (1^n, 0^{N-n})$ and for $\theta :M \to N$ a morphism in the first category, $f(\theta)$  is the restriction $M_\alpha \to N_\alpha$ of $\theta$. It is well known that $f$ is an exact functor such that \begin{center}
	$f(K_{\mu})\cong S^{\mu}$ and $f(D_{n-i} \otimes D_{i}) = M_{i}$. \end{center}Moreover, $f$ induces an isomorphism of vector spaces \[\Hom_G(D_{n-k}\otimes D_k, D_{n-t}\otimes D_t) \cong \Hom_{\mathfrak{S}_n}(M_k, M_t).\]
	
	\begin{rem}\label{rem}
		From the definition of the map $\psi_l: D_{n-k}\otimes D_k \to D_{n-t}\otimes D_t$ given in Lemma \ref{bas} and the definition of the Schur functor, it follows that the matrix of the linear map $f(\Psi_l):M_k \to M_t$ with respect to some ordering of the bases of $M_k$ and $M_t$ consisting of $k$-subsets and $t$-subsets of $X$ respectively, is the intersection matrix $U_{t,k,l}$.
	\end{rem}

In the next result we obtain the decomposition of the image of any $\mathfrak{S}_n$-map $M_k \to M_t$. \begin{cor}
	\label{main3}Consider a linear combination $A= \sum_{l=0}^{t} c_l U_{t,k,l}$ of intersection matrices, where $c_l \in \mathbb{K}$. Then the irreducible decomposition of the image of $A$ is 
	\begin{equation*}\label{irrdec3} \bigoplus_{j \in J(t,k,n)}S^{(n-j,j)},\end{equation*} where $J(t,k,n)$ is as in Theorem \ref{main1}. In particular, the rank of $A$ is equal to 
	\begin{equation*}\label{rank1} \sum_{j \in J(t,k,n)}\big(\tbinom{n}{j} - \tbinom{n}{j-1} \big).\end{equation*}
\end{cor}
\begin{proof} Suppose $\Psi \in \Hom_G(D_{n-k}\otimes D_k, D_{n-t}\otimes D_t)$ is given by $\Psi = \sum_{l=0}^{t} c_l \Psi_l$. Applying the Schur functor to Theorem \ref{main1} we obtain that the irreducible decomposition of the image of the map $f(\Psi) \in \Hom_{\mathfrak{S}_n}(M_k, M_l)$ is 
	\begin{equation*}\label{irrdec2} \Ima(f(\Psi)) = \bigoplus_{j \in J(t,k,n)}S^{(n-j,j)},\end{equation*} where $J(t,k,n)$ is as in Theorem \ref{main1}.
	From Remark \ref{rem} it follows that the irreducible decomposition of the image of $A$ is 
	$\label{irrdec3} \bigoplus_{j \in J(t,k,n)}S^{(n-j,j)}.$ In particular, the rank of the matrix $A$ is equal to 
$\label{rank1} \sum_{j \in J(t,k,n)}\big(\tbinom{n}{j} - \tbinom{n}{j-1} \big)$ according to (\ref{dimsp}).
\end{proof}
\begin{rems}(1)
Corollary  \ref{main3} generalizes \cite[Corollary 14]{GGM} that proves an expression for the rank of the individual intersection matrices $U_{t,k,l}$ by completely different methods. (2) In Corollary \ref{main3} consider $l=t$, $c_0=\cdots=c_{t-1}=0$ and $c_t=1$. Then $A=U_{t,k,t}=W_{t,k}$ is the inclusion matrix. Let $ j \in \{0,1,\dots,t\}$. Then from Definition \ref{lambda} we have
\[\lambda_j(t,k,n;t) =\sum_{s=0}^{j}(-1)^{j-s}\tbinom{j}{s}\tbinom{k-s}{t-s}\tbinom{n-k-j+s}{-j+s}=\tbinom{k-j}{t-j},\]
where in the last equality we used our convention on binomial coefficients that $\tbinom{n-k-j+s}{-j+s}=0$ if $-j+s<0$. Thus with the notation of Corollary \ref{main3} we obtain $J(t,k,n)=\{0,1,\dots,t\}$ and the rank of the matrix $A=W_{t,k}$ is equal to \begin{equation*} \sum_{j \in J(t,k,n)}\big(\tbinom{n}{j} - \tbinom{n}{j-1} \big)=\big(\tbinom{n}{0} - \tbinom{n}{-1} \big)+\big(\tbinom{n}{1} - \tbinom{n}{0} \big)+\cdots +\big(\tbinom{n}{t} - \tbinom{n}{t-1} \big)=\tbinom{n}{t}.\end{equation*}
In other words we obtain that the inclusion matrix $W_{t,k}$ has maximal rank. This is a well known result \cite{Go, Wi, Ca, FeHu}.
\end{rems}

\section{Acknowledgments}
We are grateful to the referee for carefully reading the manuscript and for providing useful suggestions.

\end{document}